\documentclass[12pt]{amsart}
\usepackage[dvips]{graphicx}
\usepackage{a4wide}
\usepackage{amsmath}
\usepackage{euscript}
\usepackage{amssymb}
\usepackage[shortlabels]{enumitem}
\usepackage{amsthm}
\usepackage{amsopn}
\usepackage{multicol}
\usepackage{color}
\usepackage{mathdots}
\usepackage{multirow}
\usepackage[usenames,dvipsnames]{xcolor}
\usepackage[margin=0.8in]{geometry}

\usepackage{graphicx}
\usepackage{caption}
\usepackage{tikz}
\usepackage{booktabs}

\usepackage{caption}
\usepackage{subcaption}

\usepackage{soul}
\usepackage{mathrsfs}
\usepackage{amsfonts}
\usepackage[fixamsmath]{mathtools}
\usepackage[Bourbaki-arrow]{dynkin-diagrams}
\usepackage{mathdots}
\usepackage{stmaryrd}  
\usepackage{mathdots}

\usepackage{latexsym} 
\usepackage{amsthm}
\usepackage{stackrel}
\usepackage{amscd} 
\usetikzlibrary{decorations.markings}
\usetikzlibrary{shapes,arrows}
\usetikzlibrary{arrows.spaced}
\usepackage{tikz-cd}
\usepackage{hyperref}

\hypersetup{%
    pdftitle={  },
    pdfauthor={Claire Frechette, Solomon Friedberg, and Henrik P. A. Gustafsson},
    linktocpage=true,           
    colorlinks=true,            
    linkcolor=blue!75!black,    
    citecolor=green!50!black,   
    filecolor=magenta,          
    urlcolor=blue!75!black      
}

\setlist[itemize]{label={$\vcenter{\hbox{\tiny$\bullet$}}$}, leftmargin=2\parindent, itemsep=0.5em}

\usepackage[protrusion=true,expansion=true]{microtype}
\usepackage{soul}

\newcounter{Todo}

\newcommand{\zz}{\mathbf{z}}
\newcommand{\nd}{\mathrm{nd}}
\newcommand{\iso}{\cong}
\DeclareMathOperator{\SL}{SL}
\DeclareMathOperator{\SU}{SU}

\newenvironment{smallcases}{\left\{\begin{smallmatrix*}[l]}{\end{smallmatrix*}\right.}

\newtheorem{Thm}{Theorem}[section]	
\newtheorem{Prop}[Thm]{Proposition}	
\newtheorem{Lemma}[Thm]{Lemma}	
\newtheorem{Cor}[Thm]{Corollary}

\theoremstyle{definition} 
\newtheorem{Def}[Thm]{Definition}	
\newtheorem{Remark}[Thm]{Remark}
\newtheorem{Ex}[Thm]{Example}

\begin{document}

\title{Iwahori Whittaker functions for unramified quasi-split groups}
\author{Claire Frechette}
\address{Department of Mathematics and Mathematical Statistics, Umeå University, SE-901 87 Umeå, Sweden}
\email{claire.frechette@umu.se}

\author{Solomon Friedberg}
\address{Department of Mathematics, Boston College, Chestnut Hill, MA 02467, USA}
\email{solomon.friedberg@bc.edu}

\author{Henrik P. A. Gustafsson}
\address{Department of Mathematics and Mathematical Statistics, Umeå University, SE-901 87 Umeå, Sweden}
\email{henrik.gustafsson@umu.se}

\begin{abstract}
Let $\mathbf{G}$ be a connected unramified quasi-split reductive
group defined over a nonarchimedean local field $F$. 
This paper concerns the Iwahori Whittaker functions attached to the unramified principal series of $\mathbf{G}(F)$.
We introduce Demazure--Whittaker operators and show that in the non-split case they 
generate a Hecke algebra with unequal parameters that are determined by $\mathbf{G}$. 
We then connect these operators to the Whittaker functions, establishing formulas for calculating all values for a basis of Iwahori Whittaker functions and giving a new proof of the Casselman--Shalika formula
for the spherical Whittaker function. In doing so we  
build on the work for split groups
of Brubaker, Bump and Licata; the new feature here is the appearance of a Hecke algebra with 
unequal parameters.
\end{abstract}


\maketitle

\section{Introduction}

Let $\mathbf{G}$ be a connected unramified quasi-split reductive
group defined over a nonarchimedean local field $F$.   
Let $K$ be the hyperspecial maximal compact subgroup of $\mathbf{G}(F)$.  If $\pi$ is an unramified principal series, 
then all values of the spherical (with respect to $K$) Whittaker function for $\pi$ are easily determined from its 
values on the maximal split torus $A$ of $\mathbf{G}(F)$.  Those values are given by the Casselman--Shalika formula \cite{CS80}.  
Using the Iwahori decomposition it is natural to decompose the spherical Whittaker function for $\pi$ as a sum of finer objects, 
the Iwahori Whittaker functions. This paper computes all values for a basis of Iwahori Whittaker functions for $\pi$.

In a different direction, in order to study unipotent representations in the quasi-split case, 
Lusztig \cite[Section 4]{Lus89} introduced Demazure operators acting on the ring of regular functions on the dual of $A$,  
$\mathcal{O}(\hat{A})$, and he developed the notion of Hecke algebras with unequal parameters \cite{Lus03} to encapsulate their properties. 
See also Macdonald \cite{Mac03}.  

The goal of this note is to join these topics.  Here we define Demazure--Whittaker operators $T_i$ arising from the interwining operators attached to the simple roots operating on the unramified principal series and develop their properties.  
In particular we show directly that when $\mathbf{G}$ is quasi-split but not split they form a Hecke algebra with unequal parameters
that are determined by $\mathbf{G}$ (Theorem~\ref{thm:hecke-algebra}).
We then connect these operators to the values of the Iwahori Whittaker functions as follows.
From the Iwahori decomposition, it suffices to determine these values on the cosets $Aw$ where $w\in K$ runs over a set of
representatives for the Weyl group.  
We use the Demazure--Whittaker operators 
to establish recursive formulas for these values (Theorem~\ref{thm:recursion}). 
Following this, we define classical Demazure operators $D_i$ as the quotient of $T_i+1$ by a normalizing factor.  We  give their relations,
and then use them to give an independent proof of the Casselman--Shalika formula. 

The relationship between Demazure operators and Whittaker functions for split groups has been developed in the work of 
Brubaker, Bump and Licata \cite{BBL} and for covers of such groups in Chinta, Gunnells and Pusk\'as \cite{CGP}.  In this paper 
we show how the method may be extended to the quasi-split case by taking into account the structure of the unequal-parameter 
Hecke algebras that appear.  

The original proof of Casselman--Shalika worked with the Casselman basis (\cite{C80}, \S3), which is
dual to the one used here with respect to a certain functional based on intertwining operators; 
the transition matrix between that basis and the one used here is complicated and involves Kazhdan--Lusztig polynomials (Reeder \cite{Re}, \S5).
A formula for the 
Iwahori Whittaker functions for the Casselman basis (in the split case) was given by Reeder \cite{Re}.
Another proof of the Casselman--Shalika formula for quasi-split groups, using different methods, was given
by Gurevich and Karasiewicz~\cite{GK}.

One can also model these Whittaker functions, including the Iwahori Whittaker functions,
by solvable lattice models and relate them to representations of quantum groups.  We will give this description in a separate paper
that is joint with Carl Westerlund. Our goal in this note is to indicate that a body of work for split groups may be extended to the quasi-split case by taking Lusztig's construction into account.  We remark that it seems likely that the connection between unique functionals for $p$-adic groups and Hecke algebra modules \cite{BBL} can be ported over as well.

In Section~\ref{sec:groups} we give group-theoretic preliminaries.  In Section~\ref{sec:IW} we define the Iwahori Whittaker functions and give some 
of their properties. In Section~\ref{sec:Ti} we introduce intertwining operators and define the Demazure--Whittaker operators.
We explain how these operators allow one to evaluate the Iwahori Whittaker functions recursively, and we
prove that these operators generate an unequal parameter Hecke algebra. In Section~\ref{sec:Di}
we pass from the Demazure--Whittaker operators to the classical Demazure operators, develop their properties,
and use them to give a new proof of the Casselman--Shalika formula.
\\

\noindent \emph{Acknowledgments:} Claire Frechette is supported by NSF grant DMS-2203042, Kempe Foundation grant JCSMK24-0117, and KVA grant MA2025-0079. Solomon Friedberg is supported by NSF grant DMS-2401309. The authors thank Jayce Getz
and Spencer Leslie for helpful conversations.

\section{Group-Theoretic Preliminaries}\label{sec:groups}

Let $\mathbf{G}$ be a connected reductive group defined over a nonarchimedean local field $F$ which is unramified, i.e.,\
$\mathbf{G}$ is quasi-split over $F$ and split over an unramified extension $E$ of $F$.  Let $\mathbf{B}$ be a minimal parabolic subgroup of $\mathbf{G}$,
$\mathbf{N}$ the unipotent radical of $\mathbf{B}$, $\mathbf{N}^-$ the opposite unipotent radical, and 
$\mathbf{A}\subset \mathbf{B}$ be a maximal $F$-split torus.
Let $\mathbf{M}=Z_\mathbf{G}(\mathbf{A})$ be the centralizer of $\mathbf{A}$ in $\mathbf{G}$. It is a maximal torus since $\mathbf{G}$ is quasi-split. 
We denote the corresponding groups of $F$-points as $G$, $B$, $N$, $A$ and $M$. 
We will sometimes write $N^+$ explicitly for $N$ when contrasting with $N^-$.

Let $\Sigma$ denote the root system of $\mathbf{G}$ with respect to $\mathbf{A}$, which is sometimes called the relative root system,
and let $W$ denote the Weyl group of $\Sigma$.
The root system $\Sigma$ need not be reduced.  Let $\Sigma^{\nd} = \left\{ \alpha \in \Sigma, \frac{\alpha}{2} \not\in \Sigma
\right\}$ be the set of non-divisible roots in $\Sigma$, which is a reduced root system.
We will denote the set of positive roots in $\Sigma^\nd$ by $\Sigma^\nd_+$.

Let $\mathfrak{o}$  be the ring of integers of $F$, $\mathfrak{p}$ be the prime idea of $\mathfrak{o}$,  $\varpi $ be a uniformizer of $\frak{o}$, and $q=|\mathfrak{o}/\mathfrak{p}|$.

Let $K$ be a hyperspecial maximal compact subgroup of $G$ which is in good position with respect to $\mathbf{A}$. Then, with $A_\mathfrak{o} := A \cap K$ and $M_\mathfrak{o} := M \cap K$, we have that $A/ A_\mathfrak{o} \iso M / M_\mathfrak{o}$ and $M = A M_\mathfrak{o}$. 
For each element in the Weyl group $W$, choose a representative in $K$. This choice of representatives will not affect any subsequent arguments since the functions considered are induced from unramified characters. Let $w_0$ be the long Weyl element and, by abuse of notation, its representative in $K$.

Let $J_+\subset B$ be the standard Iwahori subgroup of $K$, i.e.,\ the preimage of $B$ with respect to the canonical projection modulo $\mathfrak{p}$.
As detailed below, it will be more convenient to work with $J = w_0 J_+ w_0$ instead. Let $\delta$ be the modular quasicharacter of $B$. 

Let $\Lambda := X_*(\mathbf{A}) \iso X^*(\widehat{\mathbf{A}})$ where $\widehat{\mathbf{A}}$ is the dual split torus.
If $n$ is the rank of $\mathbf{A}$ then, $\widehat{\mathbf{A}}(\mathbb{C}) \iso (\mathbb{C}^\times)^n$ which we parametrize by $\mathbf{z} = (z_1,...,z_n)$ with $z_i \in \mathbb{C}^\times$. 
For $\lambda \in \Lambda$ we write $\mathbf{z}^\lambda := \lambda(\mathbf{z}) \in \mathbb{C}^\times$, the application of $\lambda \in X^*(\widehat{\mathbf{A}})$ on the dual complex torus element $\mathbf{z} \in \widehat{\mathbf{A}}(\mathbb{C})$.
Let $\mathbb{C}[\Lambda]$ be the ring of Laurent polynomials over $\mathbb{C}$ spanned by the monomials $\zz^\lambda$ with $\lambda \in \Lambda$ and let $\mathbb{C}(\Lambda)$ be the field of fractions.
Here we abuse notation and think of $\zz$ as both formal variables and elements of $\widehat{\mathbf{A}}(\mathbb{C})$.
In fact, we may identify $\mathbb{C}[\Lambda]$ with the ring of regular functions on $\widehat{\mathbf{A}}(\mathbb{C})$ and $\mathbb{C}(\Lambda)$ with the field of rational functions.
The Weyl group $W$ acts on $\mathbb{C}[\Lambda]$ by $w \cdot \zz^\lambda = \zz^{w \lambda}$.

At the same time we view $\lambda$ as a cocharacter on $A$ and write $\varpi^\lambda := \lambda(\varpi) \in A$.
We use these two dual notations of $\lambda \in \Lambda$ to parametrize unramified complex characters on $M$ as follows.
A complex character $\chi$ of $M$ is called unramified if it is trivial on $M_\mathfrak{o}$.
Thus, since $M/M_\mathfrak{o} \iso A/A_\mathfrak{o} \iso X_*(\mathbf{A}) \iso X^*(\widehat{\mathbf{A}})$, we can identify $\chi$ by an element $\zz \in \widehat{\mathbf{A}}(\mathbb{C}) \cong (\mathbb{C}^\times)^n$.
Specifically, we let $\chi_\zz$ be the unramified character of $M$ defined by $\chi_\zz(\varpi^\lambda A_\mathfrak{o}) = \zz^\lambda$ for all $\lambda \in \Lambda$. 

Construct the normalized unramified principal series representation 
\[(\pi_\zz,I(\zz)) := \text{Ind}_B^G(\chi_\zz)\]
where $\chi_\zz$ is trivially inflated to a character on $B$.
Let $\Phi_\circ^\zz \in I(\zz)^K$ be the unique spherical vector whose value on the identity is 1; thus  $\Phi_\circ^\zz(bk) = \delta^{1/2}\chi_\mathbf{z}(b)$ with $b \in B$ and $k \in K$. 
 
For $\alpha \in \Sigma^{\text{nd}}$, define $q_\alpha$ and $q_{\alpha/2}$ 
as in Macdonald \cite{Mac03}, Section 3.1 (see also Casselman \cite[Remark~1.1]{C80}). 
By \cite[(3.1.1)]{Mac03}, $q_{w\alpha} = q_\alpha$ for any $w \in W$, and similarly for $q_{\alpha/2}$, so it is enough to describe $q_\alpha$ and $q_{\alpha/2}$ for simple roots $\alpha$.
Casselman--Cely--Hales \cite{CCH} is a convenient source clarifying nuances in \cite{C80} and describes them as follows. 
For a root $\alpha$ in $\Sigma^\nd$, 
\begin{equation}
  \label{eq:qalpha}
  q_{\alpha} = q^{b_\alpha} \hspace{0.5cm}  \text{and} \hspace{0.5cm} 
  q_{\alpha / 2} = \begin{cases} q^{2b_\alpha} & \text{if } 2 \alpha \in \Sigma\\ 1 & \text{otherwise},\end{cases}
\end{equation}
where $b_\alpha$ the number of connected components in the Dynkin diagram of a specific Levi subgroup $M_{[\alpha]}$ of the split form of $\mathbf{G}$ over $E$ (see \cite[p. 6171]{CCH}) defined as follows. 
The split form is described by a root system called the absolute root system and the relative root system $\Sigma$ is recovered from this by identifying Galois orbits of roots as illustrated in a Tits diagram $\mathfrak{T}$.
Two examples are presented in Figure~\ref{fig:runningex} where the two Tits diagrams show the absolute root systems $D_4$ and $A_{2n}$ together with their Galois orbits in blue, and these descend to relative root systems $\Sigma$ of types $G_2$ and $BC_n$ respectively.
The Levi subgroup $M_{[\alpha]}$ is then defined by the simple roots in the Galois orbit of $\alpha$ and $b_\alpha$ is the number of connected components of its Dynkin diagram.

Note that in the case of simple roots for groups $\SL_2$ and $\SU_3$, the definitions \eqref{eq:qalpha} recover the familiar form of \cite{CS80}. The values 
of $q_\alpha$ and $q_{\alpha/2}$ are also computed in detail in \cite{Soll} for all quasi-split reductive groups. 
The examples in Figure~\ref{fig:runningex}  illustrate the nuances intrinsic to the original Macdonald and Casselman--Shalika definitions, and as we explain below, the appearance of unequal parameter Hecke algebras in the quasi-split context. We remark that $^2A_{2n}$ is the only type of absolutely simple quasi-split group for which $q_{\alpha/2}$ is not identically 1.
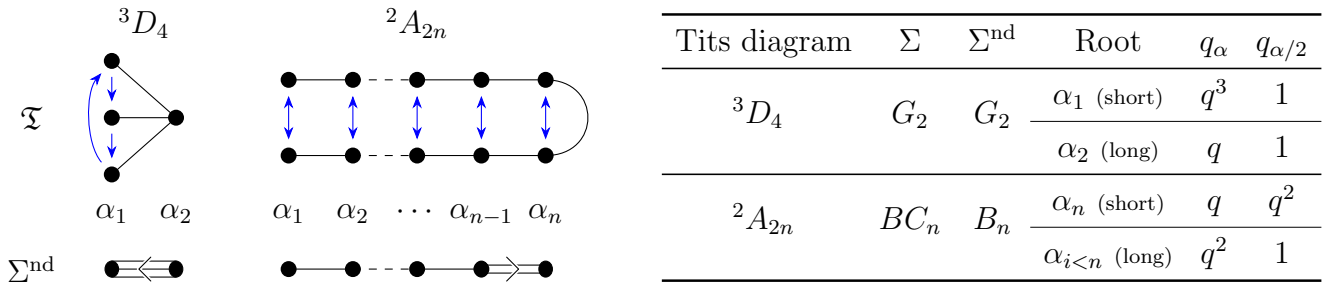
\begin{figure}[h]
\begin{tikzpicture}[xscale=0.85, baseline=-0.5cm, dot/.style = {circle, fill=black, minimum width = 6pt, inner sep=0pt}, every label/.append style = {text depth = 3pt}]
  \draw[every node/.style={dot}] (1,0) -- (0,0) node{};
  \draw[every node/.style={dot}] (0,0.75) node{} -- (1,0) node{} -- (0,-0.75) node{};
 
  \draw[-Stealth, shorten <= 6pt, shorten >= 6pt, blue] (0,0)--(0,-0.75);
  \draw[-Stealth, shorten <= 6pt, shorten >= 6pt, blue] (0,0.75)--(0,0);
  \draw[-Stealth, shorten <= 6pt, shorten >= 6pt, blue] (0,-0.75) [out=90+45,in=-90-45,looseness=1] to (0,0.75);
  \node at (0.5,1.25) {$^3D_4$};

  \node at (-1.25,0) {$\mathfrak{T}$};
  \node[below] at (0,-1) {$\alpha_1$};
  \node[below] at (1,-1) {$\alpha_2$};
  \dynkin[shift={(0,-2)}, edge length=1cm, root radius=3pt] G2;
  \node at (-1.25,0) {$\Sigma^{\nd}$};
\end{tikzpicture}
\hfill
\begin{tikzpicture}[xscale=0.85, baseline=-0.5cm, dot/.style = {circle, fill=black, minimum width = 6pt, inner sep=0pt}, every label/.append style = {text depth = 3pt}]
  \begin{scope}[shift={(0,0.5)}]
  \draw (0,0) node[dot]{} -- (1,0) node[dot]{};
  \draw (2,0) node[dot]{} -- (3,0) node[dot]{} -- (4,0) node[dot]{};
  \draw (0,-1) node[dot]{} -- (1,-1) node[dot]{};
  \draw (2,-1) node[dot]{} -- (3,-1) node[dot]{} -- (4,-1) node[dot]{};
  \draw (4,0) [out=0, in=0, looseness=2.25] to (4,-1);
  \draw[dashed] (1,0)--(2,0);
  \draw[dashed] (1,-1)--(2,-1);
  \foreach \x/\n in {0/$\alpha_1$,1/$\alpha_2$,2/$\cdots$,3/$\alpha_{n-1}$,4/$\alpha_n$}{
    \draw[Stealth-Stealth, shorten <= 6pt, shorten >= 6pt, blue] (\x,0)--(\x,-1);
    \node[below] at (\x,-1.5) {\n};
  }
\end{scope}
\node at (2,1.25) {$^2A_{2n}$};

\begin{scope}[shift={(0,-2)}]
\draw (0,0) node[dot]{} -- (1,0) node[dot]{};
\draw[dashed] (1,0) -- (2,0);
\draw (2,0) node[dot]{} -- (3,0) node[dot]{};
\dynkin[shift={(3,0)}, edge length=1cm, root radius=3pt] B{2};
\end{scope}
\end{tikzpicture}
\hfill
\begin{tabular}{cccccc}
  \toprule
  Tits diagram & $\Sigma$ & $\Sigma^\nd$ & Root & $q_\alpha$ & $q_{\alpha/2}$\\ 
  \midrule
  \multirow{2}{2em}{$^3D_4$} & \multirow{2}{*}{$G_2$} & \multirow{2}{*}{$G_2$} &$\alpha_1$ {\tiny(short)} &$q^3$& 1\\ \cmidrule{4-6}
  & & & $\alpha_2$ {\tiny(long)} &$q$&1\\ \midrule 
  \multirow{2}{2em}{$^2A_{2n}$} & \multirow{2}{*}{$BC_n$} & \multirow{2}{*}{$B_n$} &$\alpha_n$ {\tiny(short)} &$q$ &$q^2$\\ \cmidrule{4-6}
  & & & $\alpha_{i<n}$ {\tiny(long)} & $q^2$ &$1$\\
  \bottomrule
\end{tabular}
\caption{\label{fig:runningex}The Tits diagrams $\mathfrak{T}$ for $^3D_4$ and $^2A_{2n}$ along with their associated relative root system $\Sigma$, non-divisible roots $\Sigma^\nd$ and their Dynkin diagrams, and values of $q_\alpha$ and $q_{\alpha/2}$.
The Galois orbits are shown with blue arrows.}
\end{figure}

\section{Construction of the Iwahori Whittaker Functions}\label{sec:IW}

We now detail the construction of the Iwahori Whittaker functions, originally defined in \cite{CS80}, 
following the conventions used in \cite{BBBG24}. In particular, these include:
\begin{itemize}
  \item Instead of defining the Whittaker function as an integral over the upper unipotent radical~$N$, we define it as an integral over the lower unipotent $N^-$.
    Let $\psi_\mathrm{CS}$ be a principal unramified character of $N$ as defined by Casselman--Shalika in \cite[p.~209]{CS80}.
  Instead, we use the character $\psi(n) = \psi_\mathrm{CS}(w_0nw_0)$ on $N^-$.
\item Likewise, Iwahori-fixed vectors are defined with respect to the lower Iwahori subgroup $J$.
\item While \cite{CS80} defines Whittaker functions on the principal series representation $I(\zz)$, we work with Whittaker functions on the contragredient representation $I(\zz^{-1})$ induced from the upper minimal parabolic $B$.
\end{itemize}
As noted in \cite{BBBG24}, these conventions remove the long word in the Whittaker integral and simplify many calculations.

Let $\Omega_{\zz} : I(\zz) \to \mathbb{C}$ be the unique Whittaker functional that maps the spherical vector $\Phi_\circ^\zz$ to $1$ and transforms by 
$\psi$ with respect to $N^-$:
\[ \Omega_\zz(\pi_\zz(n)f)=\psi(n)\Omega_\zz(f),\qquad n\in N^{-}.\]
Then for $f \in I(\zz)$ supported in the big cell, we have \cite[Corollary~1.8]{CS80} 
\begin{equation}
  \label{eq:Whittaker-functional}
  \Omega_\zz(f) = \int_{N^-} f(n)\psi^{-1}(n)\,dn
\end{equation}
where the Haar measure on $N^-$ is normalized such that the volume of $N^- \cap K$ is 1.
By comparison with the intertwining operator,
the integral \eqref{eq:Whittaker-functional} converges absolutely in a cone in $\zz$ and may be continued to all complex parameters $\zz$.

By~\cite[Eq.~19]{C80}, the Iwahori--Bruhat decomposition gives $G=\bigsqcup_{w\in W} BwJ$.
We decompose the spherical vector $\Phi_\circ^\zz$ into Iwahori-fixed vectors supported on each Bruhat cell as $\Phi_\circ^\zz = \sum_{w \in W}\Phi_w^\zz$ defined by 
\begin{equation}
  \label{eq:Iwahori-basis}
  \Phi_w^\zz(bw'j) := \begin{cases} \Phi_\circ^\zz(b) & \text{ if } w=w'\\ 0 & \text{otherwise}\end{cases}\qquad \text{with $b\in B$,~$w'\in W$,~$j\in J$}.
\end{equation}
These form a basis of the space of Iwahori-fixed vectors $I(\zz)^{J}$ by \cite[Proposition~2.1]{C80}.
\begin{Def}
  \label{def:whittaker-function}
Define the Iwahori Whittaker function $\phi_w(\zz;g)$ for $w\in W$ as
\[ \phi_{w}(\zz; g) := \delta(g)^{1/2} \Omega_{\zz^{-1}}\bigl(\pi_{\zz^{-1}}(g)\Phi_w^{\zz^{-1}}\bigr) =  \delta(g)^{1/2} \int_{N^-} \Phi_w^{\zz^{-1}}(ng)\psi^{-1}(n)dn.\]
Similarly, define the spherical Whittaker function $\phi_{\circ}(\zz; g) := \delta(g)^{1/2}\Omega_{\zz^{-1}}\bigl(\pi_{\zz^{-1}}(g)\Phi_\circ^{\zz^{-1}}\bigr)$.
\end{Def}
Since $\Phi_\circ^\zz(g) = \sum_{w \in W} \Phi_w^\zz(g)$ by \eqref{eq:Iwahori-basis}, we have $\phi_\circ(\zz;g) = \sum_{w \in W} \phi_w(\zz;g)$.\\

Consider the $G$-dependence of these Iwahori Whittaker functions. Since $B = N M = N A M_\mathfrak{o}$ where $M_\mathfrak{o} = M \cap K \subset J$ and $w M_\mathfrak{o} w^{-1} = w M w^{-1} \cap w K w^{-1} = M \cap K = M_\mathfrak{o}$ we get that $G=\bigsqcup_{w\in W} N A M_\mathfrak{o} w J = \bigsqcup_{w\in W} N Aw J$.
Since $w_0 A w_0 = A$ we also get that $G = \bigsqcup_{w\in W} N^- A w J$.
By the above derivation we may furthermore quotient the $A$-factor by $A \cap K$ and choose representatives $\varpi^\lambda$ with $\lambda \in \Lambda$ for the cosets.
Thus, combining this decomposition with the transformation properties of the Whittaker function,  $\phi_w(\zz; g) = \psi(n) \phi_w(\zz; \varpi^\lambda w')$ where we can decompose $g = n \varpi^\lambda w' j$ with $n \in N^-$, $\lambda \in \Lambda$, $w' \in W$ and $j \in J$. 
To summarize, the Iwahori Whittaker function is completely determined by its values $\phi_w(\zz; \varpi^\lambda w')$ for $w,w' \in W$ and $\lambda \in \Lambda$ which we will compute in the remainder of the section.

The spherical Whittaker function $\phi_\circ(\zz; g) = \sum_{w \in W}\phi_w(\zz;g)$ is right-invariant under $K$ and therefore it is determined by its values $\phi_\circ(\zz; \varpi^{\lambda})$.
Because of our conventions it will be more natural to consider $\varpi^{-\lambda}$ than $\varpi^{\lambda}$ for both the spherical and Iwahori Whittaker functions.

For the remainder of this section we will compute the Iwahori Whittaker function values $\phi_{w}(\mathbf{z}; \varpi^{-\lambda} w)$ where the $w$ indexing the Iwahori-fixed vector and the $w$ in the argument are the same.
This leads to a particularly easy $p$-adic integral and we will then in Section~\ref{sec:Ti} use Deamzure--Whittaker operators to compute all values $\phi_{w}(\mathbf{z}; \varpi^{-\lambda} w')$ where the two Weyl elements can be different.

For the convenience of the reader we briefly review some aspects of Bruhat--Tits theory.
For any positive relative root $\alpha \in \Sigma_+$, let $N_\alpha$ be the subgroup of $N$ with Lie algebra $\mathfrak{g}_\alpha \oplus \mathfrak{g}_{2\alpha}$ where $\mathfrak{g}_{2\alpha} = \emptyset$ if $2\alpha \not\in \Sigma$. Similarly define $N_{-\alpha}$ as a subgroup of $N_-$.
There is a filtration $N_{\alpha,r}$ with $r \in \mathbb{R}$ such that $N_{\alpha,r} \subseteq N_{\alpha,r'}$ if and only if $r \geq r'$.
These are defined in \cite{C80} using the simply connected covering $\tilde{\mathbf{G}}$ of the derived group of $\mathbf{G}$ and the inverse image $\tilde N$ of $N$.
The paper constructs an associated Bruhat--Tits apartment $\mathcal{A}$ with a canonical homomorphism $N_G(A) \to \operatorname{Aut}(\mathcal{A})$, associated affine root system $\Sigma_\mathrm{aff}$ and finite, reduced root system $\Sigma_0$ referring to \cite{BT1, Mac03, Tits}.
This reduced Bruhat--Tits root system $\Sigma_0$ can be identified with a subset of $X^*(\mathbf{A}) \otimes_\mathbb{Z} \mathbb{R}$ which is in bijection with $\Sigma^{\nd}$ of relative roots by a map $\lambda : \Sigma \to \Sigma_0$ and in $X^*(\mathbf{A}) \otimes_\mathbb{Z} \mathbb{R}$ this map acts by a multiplication $\alpha \mapsto m_\alpha \alpha$ \cite{C80, Tamir}.
We will later (in Remark~\ref{rem:m_alpha}) show that 
\begin{equation}
  \label{eq:m_alpha}
  m_\alpha = d_\alpha = \begin{cases} 1 & \text{if } 2\alpha \not\in \Sigma \\ 2 & \text{if } 2\alpha \in \Sigma . \end{cases}
\end{equation}
Macdonald \cite{Mac03} shows that the definitions of $q_{\alpha/2} \neq 1$ and $2\alpha\not\in\Sigma$ precisely coincide. Note that these definitions are equivalent to Casselman--Shalika's notion of $\alpha$ being an $\SU_3$ root (in contrast to being an $\SL_2$ root) \cite{CS80}.

The distinction between these root systems is important because \cite{C80} defines $N_{\alpha,r}$ as the $G$-image of $\tilde N(\lambda(\alpha) + r) := \{ n \in \tilde N : nx = x \text{ for all } x \in \mathcal{A}, \lambda(\alpha)(x) + r \geq 0\}$. 
The group of translations $\Lambda \iso X_*(\mathbf{A})$ on the apartment $\mathcal{A}$ is isomorphic to $A/A_\mathfrak{o}$ via a homomorphism $\nu$ normalized such that $\nu(\varpi^{\mu})$ is the translation $-\mu$ (see for example Haines' notes \cite{Haines:intertwiners}).
The affine roots are affine functions on $x \in X_*(\mathbf{A}) \otimes_\mathbb{Z} \mathbb{R}$ so the translation $\nu(\varpi^{\mu})$ acts on $\tilde \alpha = \alpha + r$ by $(\nu(\varpi^\mu)\tilde \alpha)(x) = \tilde\alpha(\nu(\varpi^\mu)^{-1}x) = \tilde\alpha(x + \mu) = \alpha(x) + r + \alpha(\mu) = \alpha(x) + r + \langle \alpha, \mu \rangle$.
Replacing $\alpha$ by $\lambda(\alpha)$ we get
\begin{equation}
  \label{eq:muN}
  \varpi^\mu N_{\alpha,r} \varpi^{-\mu} = N_{\alpha, r +  \langle\lambda(\alpha), \mu\rangle} = N_{\alpha, r + d_\alpha \langle \alpha, \mu\rangle}.
\end{equation}
For $w$ in the finite Weyl group $W$ with representative in $K$ we get that 
\begin{equation}
  \label{eq:wN}
  w N_{\alpha,r} w^{-1} = N_{w\alpha,r}.
\end{equation}

For $w \in W$ let $J_w = w J w^{-1}$ which is also an Iwahori subgroup and we have the Iwahori factorization  (see \cite[Lemma~4.58]{Haines:Bernstein} for a convenient source which in turn references~\cite{BT2}).  
\begin{equation}
  \label{eq:Iwahori-factorization}
  J_w = (J_w \cap N^-) (J_w \cap A) (J_w \cap N^+)
\end{equation}
For the computation of the Iwahori Whittaker function we need the following description of~$J_w \cap N^-$.
\begin{Lemma}
  \label{lem:Iwahori-factorization}
  Let $w \in W$ and $J_w = w J w^{-1}$.
  Then,
  \begin{equation}
    \label{eq:JwNminus}
    J_w \cap N^- = \Bigl(\prod_{\substack{\alpha \in \Sigma^\nd_+ \\ w^{-1}\alpha > 0}} N_{-\alpha,0} \Bigr) \Bigl(\prod_{\substack{\alpha \in \Sigma^\nd_+ \\ w^{-1}\alpha < 0}} N_{-\alpha,1} \Bigr).
  \end{equation}
\end{Lemma}
\begin{proof}
  From \cite[(1.5) and (1.23)]{C80} together with \eqref{eq:wN} we have that
  \begin{equation*}
    J = (J \cap N^-) M_\mathfrak{o} (J \cap N^+) = \Bigl(\prod_{\alpha \in \Sigma^\nd_+} N_{-\alpha,0} \Bigr) M_\mathfrak{o} \Bigl(\prod_{\alpha \in \Sigma^\nd_+} N_{\alpha,1}\Bigr).
  \end{equation*}

  We also have that $N^- = \prod_{\alpha \in \Sigma^\nd} N_{-\alpha}$ where $N_{\alpha} = \bigcup_{k \in \mathbb{Z}} N_{\alpha,k}$ which together with \eqref{eq:wN} gives that
  \begin{equation*}
    w^{-1} N^- w = w^{-1} \Bigl(\prod_{\substack{\alpha > 0 \\ w^{-1}\alpha > 0}} N_{-\alpha} \Bigr)\Bigl(\prod_{\substack{\alpha > 0 \\ w^{-1}\alpha < 0}} N_{-\alpha} \Bigr) w =
    \Bigl(\prod_{\substack{w\alpha > 0 \\ \alpha > 0}} N_{-\alpha} \Bigr)\Bigl(\prod_{\substack{w \alpha > 0 \\ \alpha < 0}} N_{-\alpha} \Bigr)
  \end{equation*}
  where all roots are in $\Sigma^\nd$.
  Since we have now expressed both $J$ and $w^{-1} N^- w$ as decomposition $X = (X \cap N^-) (X \cap M_\mathfrak{o})(X \cap N^+)$ and since $N^- \cap N^+ = \{1\} = N^\pm \cap M_\mathfrak{o}$ we get that
  \begin{equation*}
    J \cap w^{-1} N^- w = \Bigl(\prod_{\substack{w\alpha > 0 \\ \alpha > 0}} N_{-\alpha,0} \Bigr)\Bigl(\prod_{\substack{w \alpha > 0 \\ \alpha < 0}} N_{-\alpha,1} \Bigr)
  \end{equation*}
  and the statement now follows from \eqref{eq:wN}.
\end{proof}

\begin{Def} Let $\lambda \in \Lambda = X_*(\mathbf{A})$ be called $w$-\emph{almost-dominant} if for all simple roots $\alpha_i \in \Sigma^{\nd}$
\begin{equation*}
  \langle \alpha_i, \lambda \rangle \geq
\begin{cases}  0 & \text{if }w^{-1}\alpha_i > 0\\
-1 &\text{if } w^{-1}\alpha_i < 0.
\end{cases}
\end{equation*}
\end{Def}

\begin{Prop}
  \label{prop:base-case} 
  For $w\in W$, let $N^-_w =  wJw^{-1} \cap N^-$. Then
  \[ \phi_{w}(\zz;\varpi^{-\lambda}w)  = \begin{cases} \operatorname{vol}(N^-_w)\cdot \zz^{\lambda} & \text{if $\lambda$ is $w$-almost-dominant}\\0 &\text{otherwise.}\end{cases}\] 
  In particular, for $w=1$,
  \[ \phi_{1}(\zz;\varpi^{-\lambda})  = \begin{cases} \zz^{\lambda} & \text{if $\lambda$ is dominant}\\0 &\text{otherwise.}\end{cases}\]
\end{Prop}

The proof takes some inspiration from \cite[Lemma~3.5, Proposition~3.6]{BBBG24}.
\begin{proof}
  We first prove that, for any Iwahori-fixed Whittaker function $\phi$ with respect to the character~$\psi$, $\phi(\varpi^{-\lambda}w)$ is zero unless $\lambda$ is $w$-almost dominant.
  For $n \in N^-$ we have that
  \begin{equation}
    \label{eq:phi-invariance}
    \psi(n)\phi(\varpi^{-\lambda}w) = \phi(n \varpi^{-\lambda}w) = \phi(\varpi^{-\lambda}w n')
  \end{equation}
  where $n' = w^{-1} \varpi^{\lambda} n \varpi^{-\lambda} w$.

  Consider the set of $n$ such that $n' \in J$.
  This is equivalent to $\varpi^{\lambda} n \varpi^{-\lambda} = w n' w \in w J w^{-1} \cap N^- = N_w^-$ which was described in Lemma~\ref{lem:Iwahori-factorization}.
  For every such $n$ we then get from \eqref{eq:phi-invariance} that $\psi(n)\phi(\varpi^{-\lambda}w) = \phi(\varpi^{-\lambda}wn') = \phi(\varpi^{-\lambda}w)$ since $n' \in J$.
  This means that for $\phi(\varpi^{-\lambda}w)$ to be non-zero we need $\psi(n) = 1$ for all $n$ such that $\varpi^{\lambda} n \varpi^{-\lambda} \in N_w^-$. 

  By Lemma~\ref{lem:Iwahori-factorization}, $N_w^- = \bigl(\prod_{\alpha >0, w^{-1}\alpha > 0} N_{-\alpha,0} \bigr) \bigl(\prod_{\alpha > 0,  w^{-1}\alpha < 0} N_{-\alpha,1} \bigr)$.
  Since $\psi$ is multiplicative it is equivalent to show that $\psi(n) = 1$ for all $\varpi^{\lambda} n \varpi^{-\lambda} \in N_{-\alpha,r_\alpha}$ for all simple roots $\alpha = \alpha_i$ of $\Sigma^\nd$ where
  \begin{equation*}
    r_\alpha = \begin{cases}
      0 & \text{if } w^{-1}\alpha > 0 \\
      1 & \text{if } w^{-1} \alpha < 0,
    \end{cases}
  \end{equation*}

  From \eqref{eq:muN} we get that $\varpi^{-\lambda} N_{-\alpha, r_\alpha} \varpi{\lambda} = N_{-\alpha, r_\alpha + d_\alpha \langle \alpha, \lambda \rangle}$ and the unramified character $\psi$ is defined in \cite{CS80} to be trivial on $N_{-\alpha, 1-d_\alpha}$ but not on $N_{-\alpha, -d_\alpha}$.

  Thus for $\phi(\varpi^{-\lambda}w)$ to be non-zero we need that $r_\alpha + d_\alpha \langle \alpha, \lambda \rangle \geq 1 - d_\alpha$ which is equivalent to
  \begin{equation*}
    \langle \alpha,\lambda \rangle \geq \frac{1-d_\alpha-r_\alpha}{d_\alpha} = 
    \begin{cases}
      -r_\alpha & \text{if } d_\alpha = 1\\
      \frac{-1-r_\alpha}{2} & \text{if } d_\alpha = 2
    \end{cases}
    = \begin{cases}
      0 & \text{if } d_\alpha = 1, w^{-1}\alpha > 0 \\
      -1 & \text{if } d_\alpha = 1, w^{-1}\alpha < 0 \\
      -\tfrac{1}{2} & \text{if } d_\alpha = 2, w^{-1}\alpha > 0 \\
      -1 & \text{if } d_\alpha = 2, w^{-1}\alpha < 0.
    \end{cases}
  \end{equation*}
  Since $\langle \alpha,\lambda \rangle$ is integral this is equivalent to $\langle \alpha,\lambda \rangle \geq \begin{smallcases} 0 & \text{if } w^{-1}\alpha > 0 \\ -1 & \text{if } w^{-1} \alpha < 0 \end{smallcases}$ irrespective of $d_\alpha$ for all simple roots $\alpha = \alpha_i$.
  We have thus shown that $\phi(\varpi^{-\lambda}w) \neq 0$ implies $\psi(\varpi^{-\lambda} N_w^- \varpi^{\lambda}) = 1$ which is equivalent to $\lambda$ being $w$-almost dominant. This concludes the first part of the proof and we will now compute $\phi_w(\mathbf{z}; \varpi^{-\lambda} w)$ for $\lambda$ $w$-almost dominant.

  By Definition~\ref{def:whittaker-function}
  \begin{equation*}
    \phi_w (\mathbf{z}, \varpi^{- \lambda} w) = \delta^{1 / 2} (\varpi^{-\lambda}) \int_{N^-} \Phi^{\mathbf{z}^{- 1}}_w (n \varpi^{- \lambda} w)\psi^{- 1} (n) d n.
   \end{equation*}
   A change of variables $n \mapsto \varpi^{-\lambda} n \varpi^{\lambda}$ multiplies the measure by $\delta(\varpi^\lambda)$ and since $\Phi^{\mathbf{z}^{-1}}_w \in I(\mathbf{z}^{-1})$ we have that $\Phi_w^{\mathbf{z}^{-1}}(\varpi^{-\lambda}n w) = (\delta^{1/2}\chi_{\mathbf{z}^{-1}})(\varpi^{-\lambda}) \Phi_w^{\mathbf{z}^{-1}}(nw)$ we get
   \begin{equation*}
     \phi_w (\mathbf{z}, \varpi^{- \lambda} w) =\mathbf{z}^{\lambda} \int_{N^-} \Phi^{\mathbf{z}^{- 1}}_w (n w) \psi^{- 1} (\varpi^{- \lambda} n \varpi^{\lambda}) d n.
   \end{equation*}
   We have that $\Phi^{\mathbf{z}^{- 1}}_w (n w) = 0$ unless $n w \in B w J$ by \eqref{eq:Iwahori-basis} where we recall that $n \in N^-$ while $B = B^+$.
   Thus, $n = b j_w^{-1}$ for some $b \in B$ and $j_w \in J_w$ which we, according to \eqref{eq:Iwahori-factorization}, can factorize as $j_w = n' b'$ with $n' \in J_w \cap N^-$ and $b' \in B$.
   Then $N^- \ni n n' = b (b')^{-1} \in B$ and since $N^- \cap B = \{1\}$ we get that $n = (n')^{-1} \in J_w \cap N^- = N_w^-$.
   Hence, $n w \in wJ$ and $b = 1$ which means that $\Phi_w^{\mathbf{z}^{-1}}(nw) = 1$ by \eqref{eq:Iwahori-basis} and
  \begin{equation*}
     \phi_w (\mathbf{z}, \varpi^{- \lambda} w) =\mathbf{z}^{\lambda} \int_{N_w^-} \psi^{- 1} (\varpi^{- \lambda} n \varpi^{\lambda}) d n.
   \end{equation*}
   We have shown that $\lambda$ being $w$-almost dominant is equivalent to $\psi(\varpi^{-\lambda} N_w^- \varpi^{\lambda}) = 1$ and are thus left with the volume integral
   \begin{equation*}
     \phi_w (\mathbf{z}, \varpi^{- \lambda} w) =\mathbf{z}^{\lambda} \int_{N_w^-} 1 \, d n = \operatorname{vol}(N_w^-) \, \mathbf{z}^\lambda.
   \end{equation*}
   
   For $w=1$, Lemma~\ref{lem:Iwahori-factorization} gives that $N_w^- = \prod_{\alpha >0} N_{-\alpha,0} = N^- \cap K$ which has volume $1$.
   The condition that $\lambda$ is $1_W$-almost dominant is equivalent to $\langle \alpha_i, \lambda \rangle \geq 0$ for all simple roots $\alpha_i$, that is, $\lambda$ is dominant. 
\end{proof}

\section{An unequal-parameter Hecke algebra of Demazure--Whittaker operators}
\label{sec:Ti}

Evaluations of Iwahori Whittaker functions for arbitrary elements of the Weyl group may be most easily computed using Demazure-type operators, which encapsulate the braiding properties of the Weyl group. In this section, we will define the operators $T_i$, derived from the action of intertwining operators on the principal series representations. We call these Demazure--Whittaker operators using the conventions of \cite{BBBG24} which treated the split case.
We will show that these generate a Hecke algebra with unequal parameters
$v_i := (q_{\alpha_i} q_{\alpha_i/2})^{-1}$ in Theorem~\ref{thm:hecke-algebra}.
The adjective \emph{unequal} is in this context traditionally interpreted as \emph{not necessarily equal}; when the group is split, the parameters $v_i$ are all equal.

We will now introduce some notation following Casselman and Casselman--Shalika \cite{C80, CS80}.
Let $\mathbf{z} \in \widehat{\mathbf{A}}(\mathbb{C})$.  For $\alpha\in\Sigma$, set
\begin{equation}
  \xi_\alpha(\zz) =
  \begin{cases}
    1-q_\alpha^{-1}\zz^{\alpha^\vee} & q_{\alpha/2}=1\\
    (1-q_\alpha^{-1}q_{\alpha/2}^{-1/2} \zz^{\alpha^\vee/2})(1+ q_{\alpha/2}^{-1/2}\zz^{\alpha^\vee/2}) & q_{\alpha/2}\neq1.
  \end{cases}
\end{equation}\label{eq:xi}
We then define 
\begin{equation}
  \eta_\alpha(\mathbf{z}) := \xi_\alpha(\mathbf{z})|_{q=1} = 1 - \mathbf{z}^{\alpha^\vee} \qquad \qquad \lambda_\alpha(\zz) := \frac{\xi_\alpha(\zz^{-1})}{\eta_\alpha(\zz)} \qquad \qquad
  c_\alpha(\zz) := \frac{\xi_\alpha(\zz)}{\eta_\alpha(\zz)}.
\end{equation}\label{eq:nlc}
(There is a small typo in \cite{CS80}, \S4, in this definition, later corrected in \cite{CCH}, regarding the $SL_2$ case.)

\begin{Remark}\label{rem:m_alpha}
We explain the connection between this and prior references.
The Casselman--Shalika formula heavily relies on two rank-one computations in simply connected covering groups attached to simple roots that are
of type $\SL_2$ and of type $\SU_3$ (\cite{CS80}, \S3).  
In \cite{C80} Casselman defines for each $\alpha \in \Sigma^\nd$ an element $a_\alpha \in A$ (which is sometimes also 
viewed as a coset in $A/A_\mathfrak{o}$) using the Bruhat--Tits construction and the associated root system $\Sigma_0$.
More precisely, the group of translations on the Bruhat--Tits apartment is isomorphic to $A/A_\mathfrak{o}$ such that $\varpi^\mu \in A$ corresponds to the translation $-\mu$. 
Casselman defines $a_\alpha$ for $\alpha \in \Sigma^\nd$ as the element or coset that correspond to the translation $-\lambda(\alpha)^\vee$ (\cite{C80}, \S1; see also \cite{Tamir}).
This is later simplified in \cite[§1.5]{CCH} using only the root system $\Sigma^\nd$ as 
\begin{equation}
  \label{eq:a_alpha}
  a_\alpha = \varpi^{\alpha^\vee/d_\alpha}  \qquad \alpha \in \Sigma^\nd.
\end{equation}
Since \eqref{eq:a_alpha} tells us that $a_\alpha$ corresponds to the translation $-\alpha^\vee/d_\alpha$ we thus get that $\lambda(\alpha)^\vee = \alpha^\vee / d_\alpha = (d_\alpha \alpha)^\vee$ giving us \eqref{eq:m_alpha}.
One may then follow the definitions of \cite{CS80} to obtain the ones given above; note that the $d_\alpha$-factor from \eqref{eq:a_alpha} merges the two cases for $\eta_\alpha$
in \cite{CS80}, \S4, into one.

Also, this recovers the $\SL_2$ and $\SU_3$ examples given in \cite[pp.~217--218]{CS80}.
  For $\SL_2(F)$ there is one positive root $\alpha$ with coroot $\alpha^\vee : t \mapsto\operatorname{diag}(t, t^{-1})$, and $d_\alpha = 1$.
  This gives that $a_\alpha = \varpi^{\alpha^\vee} = \alpha^\vee(\varpi) = \operatorname{diag}(\varpi, \varpi^{-1})$.
  For $\SU_3(F)$ we start with an absolute root system for $\SL_3$ with positive roots $\beta_1$, $\beta_2$ and $\beta_1 + \beta_2$.
  Reducing to the maximal split torus $A$ for $\SU_3(F)$ we get that $\beta_1|_{A} = \beta_2|_{A} =: \alpha$ and $(\beta_1 + \beta_2)|_{A} = 2\alpha$, so the positive relative roots $\Sigma_+ = \{\alpha, 2\alpha\}$ and $\Sigma^\nd_+ = \{ \alpha \}$ with $d_\alpha = 2$.
  The cocharacter lattice $X^*(\mathbf{A})$ is generated by $h : t \to \operatorname{diag}(t,1,t^{-1})$ parametrizing the whole split torus $A$.
  Since $\alpha(h(t)) = \beta_1(h(t)) = \beta_2(h(t)) = t$ we get that $\langle \alpha, h \rangle = 1$ and thus $h = \alpha^\vee/2$ and we get that $a_\alpha = \varpi^{\alpha^\vee/2} = \operatorname{diag}(\varpi, 1, \varpi^{-1})$.
  \end{Remark}

Recall the action of $W$ on $\mathbb{C}[\Lambda]$ where $w \mathbf{z}^\lambda = \mathbf{z}^{w\lambda} = (w^{-1}\mathbf{z})^\lambda$.
Since $q_{w\alpha} = q_\alpha$ for any $w \in W$ (\cite{Mac03}, eqn.\ (3.1.1)) 
and similarly for $q_{\alpha/2}$ we get that $\xi_{w\alpha}(\mathbf{z}) = w \xi_{\alpha}(\mathbf{z}) = \xi_\alpha(w^{-1} \mathbf{z})$.
Note also that from the definition of $\xi_\alpha$ we get that 
$$\xi_{\alpha}(\mathbf{z}^{-1}) = \xi_{-\alpha}(\mathbf{z}) = \xi_{s_\alpha \alpha}(\mathbf{z}) = \xi_{\alpha}(s_\alpha \mathbf{z})$$
where $s_\alpha$ is the reflection associated to $\alpha$.
The same identities also hold for $\eta_\alpha$, $\lambda_\alpha$ and $c_\alpha$.
We will denote the simple reflections $s_{\alpha_i}$ by $s_i$.

Let $\mathcal{A}^\zz_i : I(\zz) \to I(s_i\zz)$ denote the standard intertwining operators which is denoted by $T_{s_i}$ in~\cite{CS80}, but we will use this notation for generators of a Hecke algebra (the Demazure--Whittaker operators) in this paper.
We will now need some results about how $\mathcal{A}^\zz_i$ interacts with the Iwahori basis \eqref{eq:Iwahori-basis} and the Whittaker functional \eqref{eq:Whittaker-functional} from \cite{C80} and \cite{CS80}.
The versions we present here are modified to take our difference choices of conventions into account (basically amounting to appropriate insertions of $w_0$).

\begin{Prop}[{\cite[Theorem~3.4]{C80}}]
  \label{prop:intertwiner-basis}
  \begin{equation*}
    \mathcal{A}^\zz_i \Phi_w^\zz = 
    \begin{cases} 
      (c_{\alpha_i}(\mathbf{z}) - v_i) \Phi_w^{s_i \mathbf{z}} + \Phi_{s_i w}^{s_i \mathbf{z}} & \textnormal{if } s_i w > w\\
      (c_{\alpha_i}(\mathbf{z}) - 1) \Phi_w^{s_i \mathbf{z}} + v_i \Phi^{s_i \mathbf{z}}_{s_i w} & \textnormal{if } s_i w < w .
    \end{cases}
  \end{equation*}
\end{Prop}

\begin{Prop}[{\cite[Proposition~4.3]{CS80}}]
  \label{prop:intertwiner-functional}
  For any $f \in I(\zz)$,
  \begin{align*}
    \Omega_{s_i\zz} \circ \mathcal{A}_i(f)=\lambda_{\alpha_i}(\zz) \Omega_\zz(f).
  \end{align*}
\end{Prop}

\begin{Def}
  \label{def:Ti}
  For each simple root $\alpha_i \in \Sigma^{\nd}$ we define the \emph{Demazure--Whittaker operator} $T_i = T_{\alpha_i}$ on $\mathbb{C}(\Lambda)$ by
  \begin{equation}
    \label{eq:Ti}
    T_i := \lambda_{\alpha_i}(\mathbf{z}) s_i - (c_{\alpha_i}(\mathbf{z}) - v_i) 
  \end{equation}
  where $v_i := (q_{\alpha_i} q_{\alpha_i/2})^{-1}$.
\end{Def}
\noindent A straightforward calculation ensures that $T_i$ also restricts to a well-defined operator on $\mathbb{C}[\Lambda]$.

Continuing the examples of Figure~\ref{fig:runningex}, we provide the explicit forms of the $T_i$ operators in two cases to illustrate the new facets of the unequal-parameter Hecke algebra.

\begin{Ex}\label{D4} Let $G$ correspond to the Tits diagram $^3D_4$, which has a relative root system $\Sigma$ of type $G_2$ with two types of simple roots labelled $\alpha_1$ (short) and $\alpha_2$ (long).
See Figure~\ref{fig:runningex}. Combining \eqref{eq:Ti}, \eqref{eq:xi}, and \eqref{eq:nlc}, the explicit forms of the $T_i$ operators for $^3D_4 $ are: 
\begin{equation*}
T_1 = \frac{1-q^{-3}\zz^{-\alpha_1^\vee}}{1-\zz^{\alpha_1^\vee}}s_{1} -\frac{1-q^{-3}}{1-\zz^{\alpha_1^\vee}}
\hspace{0.5cm} \text{ and }\hspace{0.5cm}
T_2 = \frac{1-q^{-1}\zz^{-\alpha_2^\vee}}{1-\zz^{\alpha_2^\vee}}s_{2} -\frac{1-q^{-1}}{1-\zz^{\alpha_2^\vee}}
\end{equation*}
\end{Ex}

\begin{Ex}
  Similarly, consider the quasi-split group $G=\SU_{2n+1}$, which has Tits diagram $^2A_{2n}$ with relative root system $\Sigma$ of type $BC_n$, reduced root system $\Sigma^\nd$ of type $B_n$, and simple of roots $\alpha_i$ (long) for $i<n$ and $\alpha_{n}$ (short).
See Figure~\ref{fig:runningex}.
Note that $2\alpha_n \in \Sigma$ and $q_{\alpha_n/2} \neq 1$ which~\cite{CS80} denotes as a root of type $\SU_3$.
The Demazure--Whittaker operators are
\begin{align*}
T_{n} &= \frac{(1-q^{-2} \zz^{-\alpha_n^\vee/2})(1+ q^{-1}\zz^{-\alpha_n^\vee/2})}{1-\zz^{\alpha_n^\vee}}s_{n} - \frac{1-q^{3} - (q^{-2}-q^{-1})\zz^{\alpha_n^\vee/2}}{1-\zz^{\alpha_n^\vee}}\\
T_{i}&= \frac{1-q^{-2}\zz^{-\alpha_{i}^\vee}}{1-\zz^{\alpha_{i}^\vee}}s_{i} -\frac{1-q^{-2}}{1-\zz^{\alpha_{i}^\vee}}
\end{align*}
for $i < n$.
\end{Ex}

To find the inverse of $T_i$ we will use the following identities which follow straightforwardly from the definitions of $\lambda_{\alpha_i}$ and $c_{\alpha_i}$.
\begin{equation}
  \label{eq:ci}
  c_{\alpha_i} (\mathbf{z}) + c_{\alpha_i} (s_i \mathbf{z}) = 1 + v_i .
\end{equation}
\begin{equation}
  \label{eq:la2ci2}
  \lambda_{\alpha_i} (\mathbf{z}) \lambda_{\alpha_i} (s_i \mathbf{z}) + (c_{\alpha_i} (\mathbf{z}) - 1)
  (c_{\alpha_i} (\mathbf{z}) - v_i) = v_i.
\end{equation}
\begin{Lemma}
  \label{lem:inverse}
  The operator $T_i$ has the inverse
  \[ T_i^{- 1} = \frac{\lambda_{\alpha_i} (\mathbf{z}) s_i - (c_{\alpha_i} (\mathbf{z}) -
     1)}{v_i} \]
\end{Lemma}

\begin{proof}
  
  \begin{eqnarray*}
    v_i T_i^{- 1} T_i & = & (\lambda_i (\mathbf{z}) s_i - (c_i (\mathbf{z}) -
    1)) (\lambda_i (\mathbf{z}) s_i - (c_i (\mathbf{z}) - v_i))\\
    & = & - \lambda_i (\mathbf{z}) (c_i (s_i \mathbf{z}) - v_i + c_i
    (\mathbf{z}) - 1) s_i + \lambda_i (\mathbf{z}) \lambda_i (s_i \mathbf{z})
    + (c_i (\mathbf{z}) - 1) (c_i (\mathbf{z}) - v_i)
  \end{eqnarray*}
  which by \eqref{eq:ci} and \eqref{eq:la2ci2} is equal to $v_i$.
\end{proof}

These Demazure--Whittaker operators describe recursive formulas for all Iwahori Whittaker functions as follows.
\begin{Thm}\label{thm:recursion}
  Let $w \in W$ and $g \in G$.
  Then,
  \begin{equation*}    
    \phi_{s_iw}(\zz;g) = 
    \begin{cases}
      T_i \, \phi_w (\mathbf{z}; g) & \textnormal{if } s_i w > w\\
      T_i^{- 1} \phi_w (\mathbf{z}; g) & \textnormal{if } s_i w < w.
    \end{cases}
  \end{equation*}
  
\end{Thm}

\begin{proof}
To obtain the recursion relations we compute the action of $\Omega_{\mathbf{z}^{- 1}}
\circ \mathcal{A}_{i}^{s_i \mathbf{z}^{- 1}}$ on $\pi(g) \Phi_w^{s_i
\mathbf{z}^{- 1}}$ in two ways using the fact that $\mathcal{A}_i$ is a $G$-homomorphism, that is,
\begin{equation*}  
  \Omega_{\mathbf{z}^{- 1}} \circ \mathcal{A}^{s_i \mathbf{z}^{- 1}}_{i} (\pi(g) \Phi_w^{s_i \mathbf{z}^{- 1}}) = 
  \Omega_{\mathbf{z}^{- 1}} \bigl(\pi(g) \mathcal{A}_{i}^{s_i \mathbf{z}^{- 1}} (\Phi_w^{s_i \mathbf{z}^{- 1}})\bigr).
\end{equation*}
For convenience we have dropped the subscript of $\pi(g) = \pi_{s_i \zz^{-1}}(g)$.
Applying Proposition~\ref{prop:intertwiner-functional} to the left-hand side and Proposition~\ref{prop:intertwiner-basis} to the right-hand side we get
\begin{equation*}
 \lambda_{\alpha_i} (s_i \mathbf{z}^{- 1}) \Omega_{s_i \mathbf{z}^{- 1}} (\pi(g)
   \Phi_w^{s_i \mathbf{z}^{- 1}}) = \Omega_{\mathbf{z}^{- 1}} \left( \pi(g)
     \left\{\begin{array}{ll}
	 (c_{\alpha_i} (s_i \mathbf{z}^{- 1}) - v_i) \Phi_w^{\mathbf{z}^{-1}} + \Phi_{s_i
     w}^{\mathbf{z}^{- 1}} & \text{if } s_i w > w\\
     (c_{\alpha_i} (s_i \mathbf{z}^{- 1}) - 1) \Phi_w^{\mathbf{z}^{- 1}} + v_i
     \Phi^{\mathbf{z}^{- 1}}_{s_i w} & \text{if } s_i w < w
    \end{array} \right\} \right) 
\end{equation*}
that is,
\[ \lambda_{\alpha_i} (s_i \mathbf{z}^{- 1}) \phi_w (s_i \mathbf{z}; g) =
   \left\{\begin{array}{ll}
     (c_{\alpha_i} (s_i \mathbf{z}^{- 1}) - v_i) \phi_w (\mathbf{z}; g) + \phi_{s_i w}(\mathbf{z}; g) & \text{if } s_i w > w\\
     (c_{\alpha_i} (s_i \mathbf{z}^{- 1}) - 1) \phi_w (\mathbf{z}; g) + v_i \phi_{s_iw} (\mathbf{z}; g) & \text{if } s_i w < w
   \end{array}\right. \]
   Noting that $\lambda_{\alpha_i} (s_i \mathbf{z}^{- 1}) = \lambda_{\alpha_i} (\mathbf{z})$ and $c_{\alpha_i} (s_i \mathbf{z}^{- 1}) = c_{\alpha_i} (\mathbf{z})$, and solving for $\phi_{s_i w}(\mathbf{z}; g)$ we get
\[ \phi_{s_i w} (\mathbf{z}; g) = \left\{\begin{array}{ll}
     \lambda_i (\mathbf{z}) \phi_w (s_i \mathbf{z}; g) - (c_{\alpha_i} (\mathbf{z}) - v_i) \phi_w (\mathbf{z}; g) & \text{if } s_i w > w\\
     v_i^{- 1} \lambda_i (\mathbf{z}) \phi_w (s_i \mathbf{z}; g) - v_i^{- 1}(c_{\alpha_i} (\mathbf{z}) - 1) \phi_w (\mathbf{z}; g) & \text{if } s_i w < w
   \end{array}\right. \]
\end{proof}

\begin{Lemma}[Lemma 3 of \cite{BBL}]\label{lem:BBL}
  Suppose that $\mathcal{O}$ is an operator on $\mathbb{C}[\Lambda]$ on the form $\sum_{w \in W} f_w \cdot w$ where $f_w \in \mathbb{C}[\Lambda]$, and that $\mathcal{O} \zz^\lambda = 0$ for every dominant weight $\lambda \in \Lambda$. Then $\mathcal{O}=0$.
\end{Lemma}

We will now show that the $T_i$ operators braid analogously to how the split case was proven in \cite[Section~3]{BBL}.
\begin{Prop}
  \label{prop:braiding}
Let $s_i$ and $s_j$ be simple reflections in $W$. Then the corresponding Demazure--Whittaker operators $T_{i}$ and $T_{j}$ satisfy the same braiding relations as $s_i$ and $s_j$. That is to say,
if $s_is_js_i\cdots = s_js_is_j\cdots$, where there are $k$ elements on each side, then
\[ T_iT_jT_i \cdots = T_jT_iT_j\cdots\]
with $k$ elements on each side as well.
\end{Prop}

\begin{proof}
  Suppose $s_i, s_j$ are simple reflections in $W$ such that $(s_is_j)^k=1$ for a minimal $k$, so $w := s_is_js_i\cdots = s_js_is_j\cdots$ with $k$ elements on each side.
  By Theorem~\ref{thm:recursion},
  \[ T_iT_jT_i \cdots \phi_1(\zz; g) = \phi_w(\zz;g) = T_jT_iT_j\cdots \phi_1(\zz;g).\]
  From Proposition~\ref{prop:base-case} we have that $\phi_1(\zz; \varpi^{-\lambda}) = \zz^\lambda$ for dominant $\lambda \in \Lambda$ and thus 
  \begin{equation}
    \label{eq:braiding}
    (T_iT_jT_i\cdots - T_jT_iT_j\cdots )\zz^\lambda = 0
  \end{equation} 
  for all dominant $\lambda \in \Lambda$.
  We may collect the common denominator in \eqref{eq:braiding} as a function $f(\mathbf{z})$ and write \eqref{eq:braiding} as $f(z)^{-1} \mathcal{O} \mathbf{z}^\lambda = 0$ for an operator $\mathcal{O}$ on $\mathbb{C}[\Lambda]$ with coefficients in the same ring.
  The statement now follows from Lemma~\ref{lem:BBL}.
\end{proof}

Since the operators $T_i$ braid we can now define the operator $T_w$ for any $w \in W$ by $T_w = T_{i_1} \cdots T_{i_k}$ where $w = s_{i_1} \cdots s_{i_k}$ is a reduced expression.
The braiding property ensures that this does not depend on the choice of reduced expression.

We will now prove that the $T_i$ operators generate an unequal-parameter Hecke algebra, a concept originally defined by Lusztig in \cite{Lus89} and further investigated in \cite{Lus03}.

\begin{Thm}
  \label{thm:hecke-algebra}
  The operators $T_i$ generate a finite, unequal-parameter Iwahori--Hecke algebra
  with parameters $v_i = (q_{\alpha_i}q_{\alpha_i/2})^{-1}$.
  That is, the $T_i$ satisfy the same braid relations as the simple
  reflections in $W$ and they satisfy the quadratic relations $T_i^2 + (1 -
  v_i) T_i = v_i$.
\end{Thm}

\begin{proof} Proposition~\ref{prop:braiding} proves the braiding relations of the $T_i$.
  Next, we rewrite the quadratic relation as $T_i(T_i + 1 - v_i)/v_i = 1$.
From Definition~\ref{def:Ti} and Lemma~\ref{lem:inverse} it directly follows that $T_i^{-1} = (T_i + 1 - v_i)/v_i$ which completes the proof.
\end{proof}

\section{Classical Demazure operators and the Casselman--Shalika formula}
\label{sec:Di}
In this section we examine a shifted operator, the classical Demazure operators $D_i$, which may be written as a normalized form of $T_i+1$ where $T_i$ is the Demazure--Whittaker operators from Section~\ref{sec:Ti}.
The classical Demazure operators $D_i$ are known to braid and are idempotent and do not depend on the parameters $v_i$.
As an important result of this paper we show in Theorem~\ref{thm:TwDw0} that $\sum_{w \in W} T_w$ equals $D_{w_0}$ times a deformed Weyl denominator.
Using this, we will also recover an independent proof of the Casselman--Shalika formula in Theorem~\ref{thm:CS}.
First we need the following lemmas.

\begin{Lemma}
  \label{lem:Tip1}
  We have that $s_i (T_i + 1) = \frac{\xi_{\alpha_i} (\mathbf{z})}{\xi_{\alpha_i}(s_i\mathbf{z})} (T_i + 1)$.
\end{Lemma}

\begin{proof}
  First, we note that from the definition \eqref{eq:Ti} of $T_i$ together with \eqref{eq:ci} and the identity $\xi_{\alpha_i}(s_i\zz) = \xi_{\alpha_i}(\zz^{-1})$, we get that 
\begin{equation}
  \label{eq:Tip1}
  T_i + 1 = \frac{\xi_{\alpha_i} (\mathbf{z}^{-1})}{\eta_{\alpha_i} (\mathbf{z})} s_i + c_{\alpha_i}(s_i \zz) = \frac{\xi_{\alpha_i} (s_i\mathbf{z})}{\eta_{\alpha_i} (\mathbf{z})} s_i + \frac{\xi_{\alpha_i} (s_i \mathbf{z})}{\eta_{\alpha_i}(s_i \mathbf{z})}. 
\end{equation}
  Thus, when applying $s_i$, we get
\begin{equation*}
  s_i (T_i + 1) = 
  \frac{\xi_{\alpha_i} (\mathbf{z})}{\eta_{\alpha_i}(s_i\mathbf{z})} + \frac{\xi_{\alpha_i} ( \mathbf{z})}{\eta_{\alpha_i}(\mathbf{z})}s_i = 
  \frac{\xi_{\alpha_i} (\mathbf{z})}{\xi_{\alpha_i}(s_i\mathbf{z})} (T_i + 1).
\end{equation*}
\end{proof}

\begin{Cor}
  \label{cor:Tcirc}Let $T_{\circ} := \sum_{w \in W} T_w$. Then
  \[ s_i T_{\circ} = \frac{\xi_{\alpha_i} (\mathbf{z})}{\xi_{\alpha_i} (\mathbf{z}^{- 1})}
     T_{\circ} \]
\end{Cor}

\begin{proof}
  We have that
  \begin{align*} T_{\circ}  =\sum_{\substack{w \in W\\
       s_i w > w}} T_w + \sum_{\substack{w \in W\\
       s_i w < w}} T_w = \sum_{\substack{w \in W\\
       s_i w > w}} (T_w + T_{s_i w}) = \sum_{\substack{ w \in W\\
       s_i w > w}} (1 + T_{s_i}) T_w 
       \end{align*}
  since $T_{s_i w} = T_{s_i} T_w$ for $s_i w > w$. Thus, by
  Lemma~\ref{lem:Tip1},
  \[ s_i T_{\circ} = s_i \sum_{\substack{w \in W\\
       s_i w > w}} (1 + T_i) T_w = \frac{\xi_i (\mathbf{z})}{\xi_i
     (s_i\mathbf{z})} \sum_{\substack{w \in W\\
       s_i w > w}} (T_i + 1) T_w = \frac{\xi_i (\mathbf{z})}{\xi_i
     (s_i\mathbf{z})} T_{\circ} . \]
\end{proof}

\begin{Def}
  \label{def:Di}
  Let $D_i := \frac{1}{\xi_{\alpha_i} (s_i \mathbf{z})} (T_i + 1)$ which we call the classical Demazure operators. 
\end{Def}
Indeed, using \eqref{eq:Tip1} and the definition of $\eta_{\alpha}$ we get the following expression for $D_i$.
\begin{equation*}
  D_i = \frac{1}{\eta_{\alpha_i}(\mathbf{z})} s_i + \frac{1}{\eta_{\alpha_i}(s_i \mathbf{z})} = \frac{1-\mathbf{z}^{- \alpha_i^\vee } s_i}{1 -\mathbf{z}^{- \alpha_i^\vee }}.
\end{equation*}
These are exactly the Demazure operators $\partial_i$ defined in \cite[§25]{Bump}.
They are idempotent, that is, $D_i^2 = D_i$, and satisfy the same braid relations as the simple reflections in $W$ by \cite[Proposition~25.2 and Proposition~25.3 respectively]{Bump}.
Note that there is a typo in the Proposition~25.3 where $\partial_i$ has been written as $D_i$ (see \cite[§2]{BBBG:Demazure}).
Let $w = s_{i_1} \cdots s_{i_k} \in W$ be a reduced word.
Because of the braiding of $D_i$ we can define $D_w := D_{i_1} \cdots D_{i_k}$.

\begin{Lemma}
  \label{lem:Dinv}For all simple reflections $s_i$,
  \[ s_i D_{w_0} = D_{w_0} . \]
\end{Lemma}

\begin{proof}
  Let  $w = s_i w_0$. Then $s_i w > w $ and $D_{w_0} = D_i D_w$. Using idempotency,
  \[D_i D_{w_0} = D_i^2 D_w = D_i D_w = D_{w_0}.\]
  Inserting the definition of $D_i$, we see that
 
\[ (1 -\mathbf{z}^{- d_i \alpha_i} s_i) D_{w_0} = (1
  -\mathbf{z}^{- d_i \alpha_i}) D_{w_0}\]
  and thus $ s_i D_{w_0} =D_{w_0}.$
\end{proof}

\begin{Thm}\label{thm:TwDw0}
  Let $\Delta_q(\zz) = \prod_{\alpha \in \Sigma^\nd_+} \xi_{\alpha}(\mathbf{z})$.
  Then $\sum_{w \in W} T_w = \Delta_q(\zz^{-1}) D_{w_0}$.
\end{Thm}

In \cite{CS80}, Casselman and Shalika write $\zeta$ for our $\Delta_q$, but we use the latter because of its connection to a $q$-deformation of the Weyl denominator.

\begin{proof}
  Recall that $s_i$ sends $\alpha_i$ to $-\alpha_i$ while permuting the remaining positive roots.
  Since $\xi_{\alpha}(\mathbf{z}^{- 1}) = \xi_{-\alpha}(\mathbf{z})$ we thus get that 
  $$s_i \Delta_q(\zz^{-1}) = \frac{\xi_{\alpha_i}(\mathbf{z})}{\xi_{\alpha_i} (\mathbf{z}^{-1})} \Delta_q(\zz^{-1}) s_i.$$
  
  Let $\tilde{T}_{\circ} := \frac{1}{\Delta_q(\zz^{-1})} T_{\circ}$ where $T_\circ = \sum_{w \in W} T_w$. Then
  Corollary~\ref{cor:Tcirc} gives that $s_i \tilde{T}_{\circ} =
  \tilde{T}_{\circ}$, and thus $w \tilde{T}_{\circ} = \tilde{T}_{\circ}$ for
  all $w \in W$. Similarly, by Lemma~\ref{lem:Dinv} we get that $w D_{w_0} =
  D_{w_0}$ for all $w \in W$.
  
  We expand $\tilde{T}_{\circ}$ and $D_{w_0}$ into linear combinations of
  operators $w \in W$ with coefficients in $\mathbb{C}(\Lambda)$, similar to how
  $T_i$ and $D_i$ are expressed as linear combinations of $s_i$ and $1_W$, as follows:
  \[ \tilde{T}_{\circ} = \sum_{w \in W} t_w w \quad \text{ and } \quad D_{w_0} = \sum_{w \in W}
     d_w w .\]
  Since $w  \tilde{T}_{\circ} = \tilde{T}_{\circ}$ for $w \in W$ we get that
  \[\sum_{y \in W} t_y y = \tilde{T}_{\circ} = w \tilde{T}_\circ = \sum_{y \in W} w t_y y = \sum_{y \in W} {}^w (t_y) w y,\]
  where ${}^w (t_y)$ is the left-action of $w$ on $t_y \in \mathbb{C}(\Lambda)$. Thus $t_{w y} = {}^w (t_y)$ and in
  particular $t_{w'} = {}^{w' w_0} (t_{w_0})$. Similarly, $d_{w'} = {}^{w'
  w_0} (d_{w_0})$. Hence, it remains to show that $t_{w_0} = d_{w_0}$.
  
  Let $w_0 = s_{i_1} \cdots s_{i_k}$ be a reduced word. The only contribution to
  $t_{w_0}$ in $\tilde{T}_{\circ} = \frac{1}{\Delta_q(\zz^{-1})} \sum_{w \in W} T_w$ comes
  from $T_{w_0}$. Thus, since $T_i = \lambda_{\alpha_i}(\zz) s_i - (c_{\alpha_i}(\zz) - v_i)$, we get that $t_{w_0}$ equals $\frac{1}{\Delta_q(\zz^{-1})}$ times the
  product of all the $\lambda_{\alpha_{i}}(\zz)s_{i}$ terms appearing in $T_{w_0} =
  T_{i_1} \cdots T_{i_k}$
  
  Since $^{s_j}\bigl(\lambda_{\alpha_i}(\mathbf{z})\bigr) = \lambda_{s_j\alpha_i}(\mathbf{z})$ we get that
  \begin{equation*}
    t_{w_0} = \frac{1}{\Delta_q(\zz^{-1})} \lambda_{\alpha_{i_1}}\!(\zz)\, \lambda_{s_{i_1} \alpha_{i_2}}\!(\zz)\, \lambda_{s_{i_1} s_{i_2} \alpha_{i_3}}\!(\zz)\, \cdots \lambda_{s_{i_1} \cdots s_{i_{k - 1}}\alpha_{i_k}}\!(\zz) = \frac{1}{\Delta_q(\zz^{-1})} \prod_{\alpha \in \Sigma^\nd_+} \lambda_{\alpha}(\zz)
  \end{equation*}  
  where the last equality follows from the fact that these roots enumerate $\Sigma^\nd_+$ by Bourbaki \cite[Ch VI §1 Corollaire 2]{Bour02}.

  Recall that $D_i = \frac{1}{\xi_{\alpha_i}(\zz^{-1})} (T_i+1)$.
  Thus
  \[ D_{w_0} = \frac{1}{\xi_{\alpha_{i_1}}(\zz^{-1})} (T_{i_1} + 1) \frac{1}{\xi_{\alpha_{i_2}}(\zz^{-1})} (T_{i_2} + 1) \cdots \frac{1}{\xi_{\alpha_{i_k}}(\zz^{-1})}  (T_{i_k} + 1) \]
  and the coefficient $d_{w_0}$ can be read off from the term which is the product of the $\lambda_{\alpha_{i}}s_i$ terms in all the $T_{i}+1$ factors together with the prefactors $\frac{1}{\xi_{\alpha_{i}}(\zz^{-1})}$. Since $\lambda_\alpha(\zz)/\xi_\alpha(\zz^{-1}) = 1/\eta_\alpha(\zz)$ and $^{s_j} \eta_{\alpha_i} = \eta_{s_j \alpha_i}$ we get that
  \begin{equation*} 
    d_{w_0} = \frac{1}{\eta_{\alpha_{i_1}}(\zz)} \cdot \frac{1}{\eta_{s_{i_1} \alpha_{i_2}}(\zz)} \cdots \frac{1}{\eta_{s_{i_1} \cdots s_{i_{k - 1}} \alpha_{i_k}}(\zz)} 
    = \prod_{\alpha\in \Sigma^{\nd}_+} \frac{1}{\eta_{\alpha}(\zz)}
  \end{equation*}
which equals $t_{w_0}$ since $\frac{1}{\Delta_q(\zz^{-1})} \prod_{\alpha} \lambda_\alpha(\zz)= \prod_{\alpha} \lambda_\alpha(\zz) / \xi_{\alpha}(\mathbf{z}^{-1}) = \prod_{\alpha} 1/\eta_\alpha(\zz)$ where the products are over $\alpha \in \Sigma^\nd_+$.
\end{proof}

The idea of proving the operator equality in Theorem~\ref{thm:TwDw0} by showing that $t_{w_0} = d_{w_0}$ was inspired by the proof of~\cite[Theorem~3]{CGP}.

From Theorem~\ref{thm:TwDw0} we can recover the Casselman--Shalika formula for the spherical Whittaker function $\phi_\circ(\mathbf{z}, g) = \delta(g)^{1/2}\Omega_{\zz^{-1}}\bigl(\pi_{\zz^{-1}}(g)\Phi_\circ^{\zz^{-1}}\bigr) = \sum_{w \in W} \phi_w(\zz;g)$ which computes all its values.
Indeed, in Section~\ref{sec:IW} we showed that the spherical Whittaker function is completely determined by its values $\phi_\circ(\zz;\varpi^{-\lambda})$ for dominant $\lambda \in \Lambda$ and the Casselman--Shalika formula computes these values as follows.

\begin{Thm}[{The Casselman--Shalika formula~\cite[Theorem~5.4]{CS80}}] 
  \label{thm:CS}
  Let $\lambda \in \Lambda$ be a dominant weight.
  Then,
  \begin{equation*}
    \phi_\circ(\zz; \varpi^{-\lambda}) = \frac{\Delta_q(\zz^{-1})}{\Delta_1(\zz)} \sum_{w\in W} (-1)^{\ell(w)} \prod_{\substack{\alpha>0\\w\alpha <0}} \zz^{\alpha^\vee} \cdot w^{-1}w_0 \mathbf{z}^\lambda
  \end{equation*}
  where $\alpha > 0$ means that $\alpha \in \Sigma^\nd_+$, where $\Delta_q(\zz) = \prod_{\alpha \in \Sigma^\nd_+} \xi_{\alpha}(\mathbf{z})$ and $d_\alpha = \begin{smallcases} 1 & \text{if } 2 \alpha \not\in \Sigma  \\ 2 & \text{if } 2 \alpha \in \Sigma \end{smallcases}$.
\end{Thm}

By $\Delta_1(\zz)$ we mean $\Delta_q(\zz)|_{q=1}$ and since $\xi_{\alpha}(\zz)|_{q=1} = \eta_\alpha(\zz)$ we have that $\Delta_1(\zz) = \prod_{\alpha>0} 1/\eta_{\alpha}(\zz)$.
(The changes between the above formula and Theorem 5.4 of Casselman and Shalika \cite{CS80} are due to the changes in convention described in Section \ref{sec:IW}.)
We observe that the only $p$-adic integral we need to compute is the base case in Proposition~\ref{prop:base-case} which reduces to a simple volume integral.
The rest follows from the Demazure operators.

\begin{proof}
  We have that $\phi_\circ(\zz;g) = \sum_{w\in W} \phi_w(\zz;g) = \sum_{w \in W} T_w \phi_1(\zz;g)$.
  Let $\lambda \in \Lambda$ be a dominant weight.
  Then from Proposition~\ref{prop:base-case} we have that $\phi_1(\zz; \varpi^{-\lambda}) = \zz^\lambda$.
  Since $\sum_{w \in W} T_w \ = \Delta_q(\zz^{-1}) D_{w_0}$ from Theorem~\ref{thm:TwDw0} we thus get that $\phi_\circ(\zz; \varpi^{-\lambda}) = \sum_{w \in W} T_w \zz^\lambda = \Delta_q(\zz^{-1}) D_{w_0} \zz^\lambda$.

  It remains to evaluate $D_{w_0}\zz^{\lambda}$. 
  In the proof of Theorem~\ref{thm:TwDw0} we showed that $D_{w_0} = \sum_{w \in W} d_w w$ with coefficients $d_w = {}^{ww_0}(d_{w_0})$ where $d_{w_0} =  \prod_{\alpha > 0} \frac{1}{\eta_{\alpha}(\zz)}$.
  After changing the summation variable from $w$ to $w^{-1}w_0$ we get
  \begin{equation*}
    D_{w_0}\zz^{\lambda} = \sum_{w\in W} \prod_{\alpha > 0} \frac{1}{{}^{w^{-1}\!}\bigl(\eta_{\alpha}(\zz)\bigr)} \cdot w^{-1}w_0 \zz^\lambda 
    = \sum_{w\in W} \prod_{\alpha > 0} \frac{1}{\eta_{w^{-1}\alpha}(\zz)} \cdot w^{-1}w_0 \zz^\lambda.
  \end{equation*}
  We rewrite the coefficient and decompose it as follows
  \begin{equation*}
  \prod_{\alpha > 0} \frac{1}{\eta_{w^{-1}\alpha}(\zz)} 
  = \!\!\! \prod_{\alpha \in w^{-1} \Sigma^\nd_+} \frac{1}{\eta_{\alpha}(\zz)} 
  = \prod_{\substack{\alpha \in \Sigma^\nd \\ w\alpha > 0}} \frac{1}{\eta_{\alpha}(\zz)} 
  = \prod_{\substack{\alpha > 0 \\ w\alpha > 0}} \frac{1}{\eta_{\alpha}(\zz)} \cdot \prod_{\substack{\alpha < 0 \\ w\alpha > 0}} \frac{1}{\eta_{\alpha}(\zz)} 
  = \prod_{\substack{\alpha > 0 \\ w\alpha > 0}} \frac{1}{\eta_{\alpha}(\zz)} \cdot \prod_{\substack{\alpha > 0 \\ w\alpha < 0}} \frac{1}{\eta_{-\alpha}(\zz)}
  \end{equation*} 
  Since $\frac{1}{\eta_{-\alpha}(\zz)} = \frac{-\zz^{\alpha^\vee}}{\eta_\alpha(\zz)}$ we get
  \begin{equation*}
    \prod_{\alpha > 0} \frac{1}{\eta_{w^{-1}\alpha}(\zz)}
    = \prod_{\substack{\alpha > 0 \\ w\alpha > 0}} \frac{1}{\eta_{\alpha}(\zz)} \cdot \prod_{\substack{\alpha > 0 \\ w\alpha < 0}} \frac{-\zz^{\alpha^\vee}}{\eta_\alpha(\zz)}
    = \prod_{\alpha > 0} \frac{1}{\eta_\alpha(\zz)} \cdot \prod_{\substack{\alpha > 0 \\ w\alpha < 0}} (-\zz^{\alpha^\vee})
    = \frac{1}{\Delta_1(\zz)} \cdot (-1)^{\ell(w)}\! \prod_{\substack{\alpha > 0 \\ w\alpha < 0}} \zz^{\alpha^\vee}.
  \end{equation*}
\end{proof}

\bibliographystyle{abbrvurl} 
\bibliography{sources.bib}

\begin{thebibliography}{10}

\bibitem{Bour02}
N.~Bourbaki.
\newblock {\em Lie groups and {L}ie algebras. {C}hapters 4--6}.
\newblock Elements of Mathematics (Berlin). Springer-Verlag, Berlin, 2002.
\newblock Translated from the 1968 French original by Andrew Pressley.
\newblock \href {https://doi.org/10.1007/978-3-540-89394-3}
  {\path{doi:10.1007/978-3-540-89394-3}}.

\bibitem{BBBG:Demazure}
B.~Brubaker, V.~Buciumas, D.~Bump, and H.~P.~A. Gustafsson.
\newblock Colored five-vertex models and {D}emazure atoms.
\newblock {\em J. Combin. Theory Ser. A}, 178:Paper No. 105354, 48, 2021.
\newblock \href {https://doi.org/10.1016/j.jcta.2020.105354}
  {\path{doi:10.1016/j.jcta.2020.105354}}.

\bibitem{BBBG24}
B.~Brubaker, V.~Buciumas, D.~Bump, and H.~P.~A. Gustafsson.
\newblock Colored vertex models and {I}wahori {W}hittaker functions.
\newblock {\em Selecta Math. (N.S.)}, 30(4):Paper No. 78, 58, 2024.
\newblock \href {https://doi.org/10.1007/s00029-024-00950-6}
  {\path{doi:10.1007/s00029-024-00950-6}}.

\bibitem{BBL}
B.~Brubaker, D.~Bump, and A.~Licata.
\newblock Whittaker functions and {D}emazure operators.
\newblock {\em J. Number Theory}, 146:41--68, 2015.
\newblock \href {https://doi.org/10.1016/j.jnt.2014.01.001}
  {\path{doi:10.1016/j.jnt.2014.01.001}}.

\bibitem{BT1}
F.~Bruhat and J.~Tits.
\newblock Groupes r\'eductifs sur un corps local.
\newblock {\em Inst. Hautes \'Etudes Sci. Publ. Math.}, (41):5--251, 1972.
\newblock URL: \url{http://www.numdam.org/item?id=PMIHES_1972__41__5_0}.

\bibitem{BT2}
F.~Bruhat and J.~Tits.
\newblock Groupes r\'eductifs sur un corps local. {II}. {S}ch\'emas en groupes.
  {E}xistence d'une donn\'ee radicielle valu\'ee.
\newblock {\em Inst. Hautes \'Etudes Sci. Publ. Math.}, (60):197--376, 1984.
\newblock URL: \url{http://www.numdam.org/item?id=PMIHES_1984__60__5_0}.

\bibitem{Bump}
D.~Bump.
\newblock {\em Lie groups}, volume 225 of {\em Graduate Texts in Mathematics}.
\newblock Springer, New York, second edition, 2013.
\newblock \href {https://doi.org/10.1007/978-1-4614-8024-2}
  {\path{doi:10.1007/978-1-4614-8024-2}}.

\bibitem{C80}
W.~Casselman.
\newblock The unramified principal series of {${ p}$}-adic groups. {I}. {T}he
  spherical function.
\newblock {\em Compositio Math.}, 40(3):387--406, 1980.
\newblock URL: \url{http://www.numdam.org/item?id=CM_1980__40_3_387_0}.

\bibitem{CCH}
W.~Casselman, J.~E. Cely, and T.~Hales.
\newblock The spherical {H}ecke algebra, partition functions, and motivic
  integration.
\newblock {\em Trans. Amer. Math. Soc.}, 371(9):6169--6212, 2019.
\newblock \href {https://doi.org/10.1090/tran/7465}
  {\path{doi:10.1090/tran/7465}}.

\bibitem{CS80}
W.~Casselman and J.~Shalika.
\newblock The unramified principal series of {$p$}-adic groups. {II}. {T}he
  {W}hittaker function.
\newblock {\em Compositio Math.}, 41(2):207--231, 1980.
\newblock URL: \url{http://www.numdam.org/item?id=CM_1980__41_2_207_0}.

\bibitem{CGP}
G.~Chinta, P.~E. Gunnells, and A.~Pusk\'as.
\newblock Metaplectic {D}emazure operators and {W}hittaker functions.
\newblock {\em Indiana Univ. Math. J.}, 66(3):1045--1064, 2017.
\newblock \href {https://doi.org/10.1512/iumj.2017.66.6068}
  {\path{doi:10.1512/iumj.2017.66.6068}}.

\bibitem{GK}
N.~Gurevich and E.~Karasiewicz.
\newblock The twisted {S}atake transform and the {C}asselman-{S}halika formula
  for quasi-split groups.
\newblock {\em Int. Math. Res. Not. IMRN}, (14):1--32 [11148--11179 on table of
  contents], 2022.
\newblock \href {https://doi.org/10.1093/imrn/rnab346}
  {\path{doi:10.1093/imrn/rnab346}}.

\bibitem{Haines:intertwiners}
T.~J. Haines.
\newblock Intertwiners for unramified groups.
\newblock Expository note.
  \url{https://www.math.umd.edu/~tjh/intertwiners_unram_grps.pdf}.

\bibitem{Haines:Bernstein}
T.~J. Haines.
\newblock The stable {B}ernstein center and test functions for {S}himura
  varieties.
\newblock In {\em Automorphic forms and {G}alois representations. {V}ol. 2},
  volume 415 of {\em London Math. Soc. Lecture Note Ser.}, pages 118--186.
  Cambridge Univ. Press, Cambridge, 2014.

\bibitem{Lus89}
G.~Lusztig.
\newblock Affine {H}ecke algebras and their graded version.
\newblock {\em J. Amer. Math. Soc.}, 2(3):599--635, 1989.
\newblock \href {https://doi.org/10.2307/1990945} {\path{doi:10.2307/1990945}}.

\bibitem{Lus03}
G.~Lusztig.
\newblock {\em Hecke algebras with unequal parameters}, volume~18 of {\em CRM
  Monograph Series}.
\newblock American Mathematical Society, Providence, RI, 2003.
\newblock \href {https://doi.org/10.1090/crmm/018}
  {\path{doi:10.1090/crmm/018}}.

\bibitem{Mac03}
I.~G. Macdonald and A.-M. Aubert.
\newblock {\em Spherical functions on a group of {$p$}-adic type}, volume 2392
  of {\em Lecture Notes in Mathematics}.
\newblock Springer, Cham, second edition, 2026.
\newblock \href {https://doi.org/10.1007/978-3-032-15671-6}
  {\path{doi:10.1007/978-3-032-15671-6}}.

\bibitem{Re}
M.~Reeder.
\newblock On certain {I}wahori invariants in the unramified principal series.
\newblock {\em Pacific J. Math.}, 153(2):313--342, 1992.
\newblock URL: \url{http://projecteuclid.org/euclid.pjm/1102635836}.

\bibitem{Soll}
M.~Solleveld.
\newblock Parameters of {H}ecke algebras for {B}ernstein components of
  {$p$}-adic groups.
\newblock {\em Indag. Math. (N.S.)}, 36(1):124--170, 2025.
\newblock \href {https://doi.org/10.1016/j.indag.2024.04.005}
  {\path{doi:10.1016/j.indag.2024.04.005}}.

\bibitem{Tamir}
B.~Tamir.
\newblock On {$L$}-functions and intertwining operators for unitary groups.
\newblock {\em Israel J. Math.}, 73(2):161--188, 1991.
\newblock \href {https://doi.org/10.1007/BF02772947}
  {\path{doi:10.1007/BF02772947}}.

\bibitem{Tits}
J.~Tits.
\newblock Reductive groups over local fields.
\newblock In {\em Automorphic forms, representations and {$L$}-functions
  ({P}roc. {S}ympos. {P}ure {M}ath., {O}regon {S}tate {U}niv., {C}orvallis,
  {O}re., 1977), {P}art 1}, volume XXXIII of {\em Proc. Sympos. Pure Math.},
  pages 29--69. Amer. Math. Soc., Providence, RI, 1979.

\end{thebibliography}

\end{document}